\documentclass[11pt]{article}
\usepackage{amsfonts}
\usepackage{amssymb}
\usepackage{graphics}
\usepackage{amssymb,amsmath,epsfig,graphics,latexsym,psfrag}
\usepackage[english]{babel}
\usepackage[all]{xy}
\usepackage{amscd, amssymb, amsmath, amsthm, graphics}
\usepackage{amsmath,amsfonts,amsthm,amssymb}
\usepackage{latexsym,amsmath}
\usepackage{graphicx,psfrag}
\usepackage{mathrsfs}

\def\G{(\Gamma,{\cal G})}
\def\Z{{\Bbb Z}}

\newtheorem{theorem}{Theorem}[section]
\newtheorem{corollary}[theorem]{Corollary}
\newtheorem{lemma}[theorem]{Lemma}
\newtheorem{example}[theorem]{Example}

\begin{document}
\title{Extending finite group actions on surfaces over
$S^3$}
\author{Chao Wang, Shicheng Wang, Yimu Zhang\\ and \\Bruno Zimmermann}
%\subjclass{57M50, 51H20}

\date{}
\maketitle

\centerline{Department of Mathematics} \vskip 0.3truecm
\centerline{Peking University, Beijing, 100871 CHINA}
\centerline{and}
 \centerline{ University of Trieste, Trieste, 34127 Trieste
ITALY}

\begin{abstract} Let $OE_g$
(resp. $CE_g$ and   $AE_g$) and resp. $OE^o_g$ be the maximum order
of finite (resp. cyclic and abelian) groups $G$ acting on the closed
orientable surfaces $\Sigma_g$ which extend over $(S^3, \Sigma_g)$
among all embeddings $\Sigma_g\to S^3$ and  resp. unknotted
embeddings $\Sigma_g\to S^3$.

It is known that $OE^o_g\le 12(g-1)$, and we show that $12(g-1)$ is
reached for an unknotted embedding $\Sigma_g \to S^3$ if and only if
$g = 2$, 3, 4, 5, 6, 9, 11, 17, 25, 97, 121, 241, 601. Moreover
$AE_g$ is $2g+2$; and $CE_g$ is $2g+2$ for even $g$, and $2g-2$ for
odd $g$.

Efforts are made to see intuitively how these maximal symmetries are
embedded into the symmetries of the 3-sphere\footnote{Keywords:
finite group action, extendable action, symmetry of surface,
symmetry of 3-sphere, maximum order.}.
\end{abstract}

\tableofcontents

\section{Introduction}
In this article, we use $\Sigma_g$ to denote the orientable closed
surface of genus $g>1$, and use $V_g$ to denote the handlebody of
genus $g>1$. All group actions will be faithful and
orientation-preserving (on both surfaces and 3-manifolds).

Let  $O_g$  (resp. $C_g$ and   $A_g$)   be the maximum order of all
 finite  (resp. cyclic and abelian) groups $G$ which can act  on
$\Sigma_g$. A classical result of Hurwitz states that $O_g$ is at
most $84(g - 1)$, proved by applying the Riemann-Hurwitz formula.
However for each fixed $g$, $O_g$ is hard to determine in general,
see \cite{Ac}, \cite{Ma} for some partial results. On the other hand
$C_g$ is $4g+2$, see \cite{St}, also \cite{Ha}, and \cite{Wa} for
more direct constructions; and $A_g$ is $4(g+1)$, see \cite{Ma}.

Let  $OH_g$  (resp. $CH_g$ and   $AH_g$)   be the maximum order of
all
 finite  (resp. cyclic and abelian) groups $G$ which can acts  on
the handlebody $V_g$. By definition we have that $O_g\ge OH_g,
C_g\ge CH_g$ and $A_g\ge AH_g.$ It is a result due to Zimmermann
\cite{Zi1} that $OH_g\le 12(g-1)$, see also \cite{Zi2}, \cite{MMZ},
\cite{MZ}. A handlebody orbifold theory was derived in \cite{MMZ},
see also \cite{Zi3} for a more geometric approach, and then proved
that $CH_g$ is $2g+2$ when $g$ is even, and $2g$ when $g$ is odd
\cite{MMZ}. Moreover $OH_g$ is bounded below by $4(g+1)$, and $OH_g$
is either $12(g-1)$ or $8(g-1)$ if $g$ is odd, and each of these is
achieved by infinitely many odd $g$ \cite{MZ}.

In the present article, we consider the following extension
problems: Suppose a finite group $G$ acts  on the surface
$\Sigma_g$. If there is an embedding $e: \Sigma_g\subset S^3$ such
that $G$ can act on the pair $(S^3, \Sigma_g)$ and the restriction
to $\Sigma_g$ is the given $G$ action on $\Sigma_g$, we call the
action of $G$ on $ \Sigma_g$  extendable (over $S^3$ with respect to
$e$).

Call an embedding $e_o:  \Sigma_g\to S^3$ unknotted if each
component of $S^3\setminus  \Sigma_g$ is a handlebody. So each
extendable $G$ action w.r.t. an unknotted embedding $e_o$ provides a
$G$-invariant Heegaard splitting of $S^3$.
%Call an embedding $e:  \Sigma_g\to S^3$ semi-unknotted if as least one
%component of $S^3\setminus  \Sigma_g$ is a handlebody.
Similarly we define an action of $G$ on $V_g$ to be extendable, and
an embedding $e_o:V_g \to S^3$ to be unknotted if the complement
$S^3\setminus V_g$ is also a handlebody. For each $g$, an unknotted
embedding is unique up to isotopy of $S^3$ and automorphisms on
$\Sigma_g$ (resp. $V_g$).

In such extension problems,
%a primary question must be to determine
%when a finite group action of $G$ on $\Sigma_g$ (resp. $V_g$) is
%extendable, which could be very difficult in general (see Example
%6.1); and
we first study  the maximum orders in the present paper. Let $OE_g$
(resp. $CE_g$ and   $AE_g$)   be the maximum order of all extendable
finite (resp. cyclic and abelian) groups $G$ acting on $\Sigma_g$.
It is not obvious that $OH_g\ge OE_g,  CH_g\ge CE_g$ and $AH_g\ge
AE_g.$
%since it is not clear yet that an action of $G$ on $
%\Sigma_g$ is extendable implies that it is extendable w.r.t some
%(semi-)unknotted embedding.
So finer notions may be useful at the moment: Let $OE^o_g$  (resp.
$CE^o_g$ and   $AE^o_g$)   be the maximum order of a finite  (resp.
cyclic and abelian) group $G$ acting on $\Sigma_g$ which extends
over $S^3$ w.r.t. an unknotted embedding. Then

(1) $OE_g\ge OE^o_g, CE_g\ge CE^o_g$ and $AE_g\ge AE^o_g$.

(2) $OH_g\ge OE^o_g,  CH_g\ge CE^o_g$ and $AH_g\ge AE^o_g.$

%For the convenient of later arguments, we state Lemma \ref{primary}
%below, where (1) and (2) follows from the definitions, (3) and (4)
%are the restatement of the results in \cite{Zi1}, \cite{MMZ} and
%\cite{MZ}.

%\begin{lemma}\label{primary}

%(3) $CH_g$ is $2g+2$ when $g$ is even, and is $2g$ when $g$ is odd.

%(4) $4(g+1)\le OH_g\le 12(g-1)$, and $OH_g$ is either $12(g-1)$ or
%$8(g-1)$ if $g$ is odd.
%\end{lemma}
Now we are going to describe the results and the organization of the
paper.

Even to determine $OE_g$ and  $OE^o_g$ are harder, some discussions
are made in Section 2. It is clear $OE^o_g\le 12(g-1)$. We show that
there are only finitely many $g$ such that $OE^o_g=12(g-1)$, and
indeed we list all such $g$ (Theorem \ref{realzation of 12(g-1)},
see also the Examples in Section 4). It is derived that for each
$g$, $4(g+1)$ is a lower bound for $OE^o_g$ (Example 4.3), and for
each $g=n^2$ a lower bound of $OE^o_g$ is $4(n+1)^2$ which is larger
than $4(g+1)$ (Example 4.4).

In Section 3, we discuss  the abelian case and the cyclic case which
are easier. By applying the handlebody orbifold theory, we will
first derive the needed information about orders of  large abelian
and cyclic group action on $V_g$ (Theorem \ref{abelian1} and Theorem
\ref{cyclic2}). Then we show that $AE_g$ is $2g+2$ (Theorem
\ref{abelian}), and $CE_g$ is $2g+2$ for even $g$, and $2g-2$ for
odd $g$  (Theorem \ref{cyclic}). All these maximum order group
actions are realized by unknotted embedding (Examples 4.1 and 4.2),
hence $CE^o_g=CE_g$ and $AE^o_g=AE_g$.

\noindent\textbf{Question 1.} {\it If an embedding $\Sigma_g\to S^3$
realizes $OE_g$, should  the embedding be unknotted? Weakly does
$OE_g=OE^o_g$  for each $g>1$?}

The existence of extendable group actions on surfaces with large
symmetry presented in Sections 2 and 3 are mostly derived from the
orbifold theory.  On the other hand in the process of this work,
most large symmetries in Sections 2 and 3 are first constructed in a
more direct and intuitive way without using orbifold theory. Section
4 presents those constructions which show us how those symmetries on
surfaces stay in the symmetries on 3-sphere. A reason of doing so is
given in the next paragraph.

We end the introduction by the some motivations of our study:
Surfaces (as well as handlebodies) are very familiar  subjects to
us, mostly because  we can see them staying in our 3-space in
various manners. The symmetries of the surfaces have been studied
for a long time, and it will be natural to wonder when these
symmetries can be embedded into the symmetries of our 3-space
(3-sphere). Another inspiring fact is a related problem on extending
surface automorphisms over 4-space which had been addressed  30
years ago \cite{Mo}, and for recent developments see \cite{Hi},
\cite{DLWY} and \cite{LNSW}.

\bigskip\noindent\textbf{Acknowledgement}. The first author is supported by Beijing International Center for Mathematical Researchㄛ Peking University.
The second author is partially supported by grant No.10631060 of the
National Natural Science Foundation of China.

\section{Maximum order of extendable group action}

From now on all groups in this paper will be finite.

A simple  picture we should keep in mind is the following: Suppose
that the action of $G$ on $ \Sigma_g$ is extendable w.r.t. some
embedding $\Sigma\subset S^3$. Let $\Gamma=\{x\in S^3|\,\exists
\,g\in G, s.t.\,gx=x\}$; then $\Gamma$ is a graph, possibly
disconnected, and $S^3/G$ is a 3-orbifold whose singular set
$\Gamma/G$ is also a graph. Each edge of $\Gamma/G$ can be labeled
by an integer $>0$ which corresponds to the singular index of it.
Also, $ \Sigma_g/G$ is a 2-orbifold with singular set $
\Sigma_g/G\cap \Gamma/G$, which are isolated points.

We recall the handlebody orbifold theory from \cite{MMZ},
\cite{Zi3}.

Let $G$ be a finite group acting on a handlebody of genus $g$.
Associated to this action there is a handlebody orbifold $V_g/G$, a
finite graph of finite groups $\G$ and a surjection $\phi :\pi_1\G
\longrightarrow   G$ whose kernel is isomorphic to a free group of
rank $g$; in particular, $\phi$ is injective on the finite vertex
groups of $\G$. Here $\pi_1\G$ denotes the fundamental group of the
graph of groups or of the corresponding handlebody orbifold:  this
is the iterated free product with amalgamation and HNN-extension
over the vertex groups, amalgamated over the edge groups of a
maximal tree, with the HNN-generators corresponding to the edges in
the complement of the chosen maximal tree. We denote by
$$ \chi = \chi\G = \sum 1/|G_v| - \sum 1/|G_e|\qquad (2.1)$$
the Euler characteristic of the graph of groups $\G$ (the sum is
taken over all vertex groups $G_v$ resp. edge groups $G_e$ of $\G$);
then
$$g-1 =  |G|  (-\chi) \qquad (2.2).$$

The vertex groups $G_v$ belongs to one of the following five classes
which correspond to the five types of finite subgroups of the
orthogonal group $SO(3)$ given in Figure 1. The edge groups $G_e$
are cyclic groups which are either trivial or maximally cyclic in
the adjacent vertex groups. We can also assume that the edge group
of an edge which is not a loop (not closed) does not coincide with
one of the two vertex groups of the edge.

\begin{center}
\scalebox{0.6}{\includegraphics*[0pt,0pt][574pt,131pt]{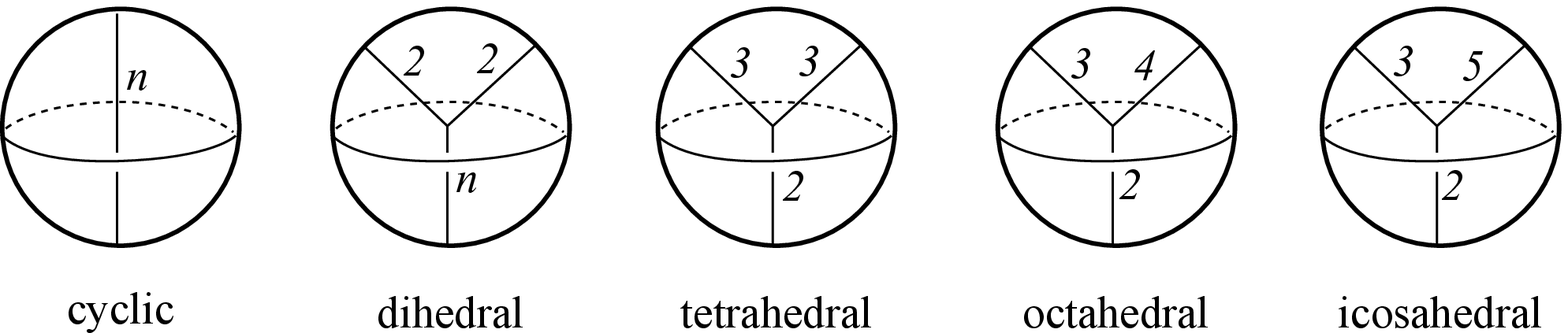}}

Figure 1
\end{center}

Conversely, to each such finite graph of finite groups associated to
a handlebody orbifold and surjection $\phi: \pi_1\G \to G$, there is
a corresponding action of $G$ on a handlebody $V_g$ of genus $g$.

\begin{theorem}\label{realzation of 12(g-1)}
Suppose the $G$-action on $ \Sigma_g$ is extendable over $S^3$
w.r.t. the unknotted embedding $e_o$ and the order of $G$ is
$12(g-1)$. Then $g$ is as follows: g = 2, 3, 4, 5, 6, 9, 11, 17, 25,
97, 121, 241, 601.
\end{theorem}

\begin{proof}
The values of $g$ can be obtained as follows.

\begin{center}

\scalebox{0.5}{\includegraphics*[0pt,0pt][650pt,459pt]{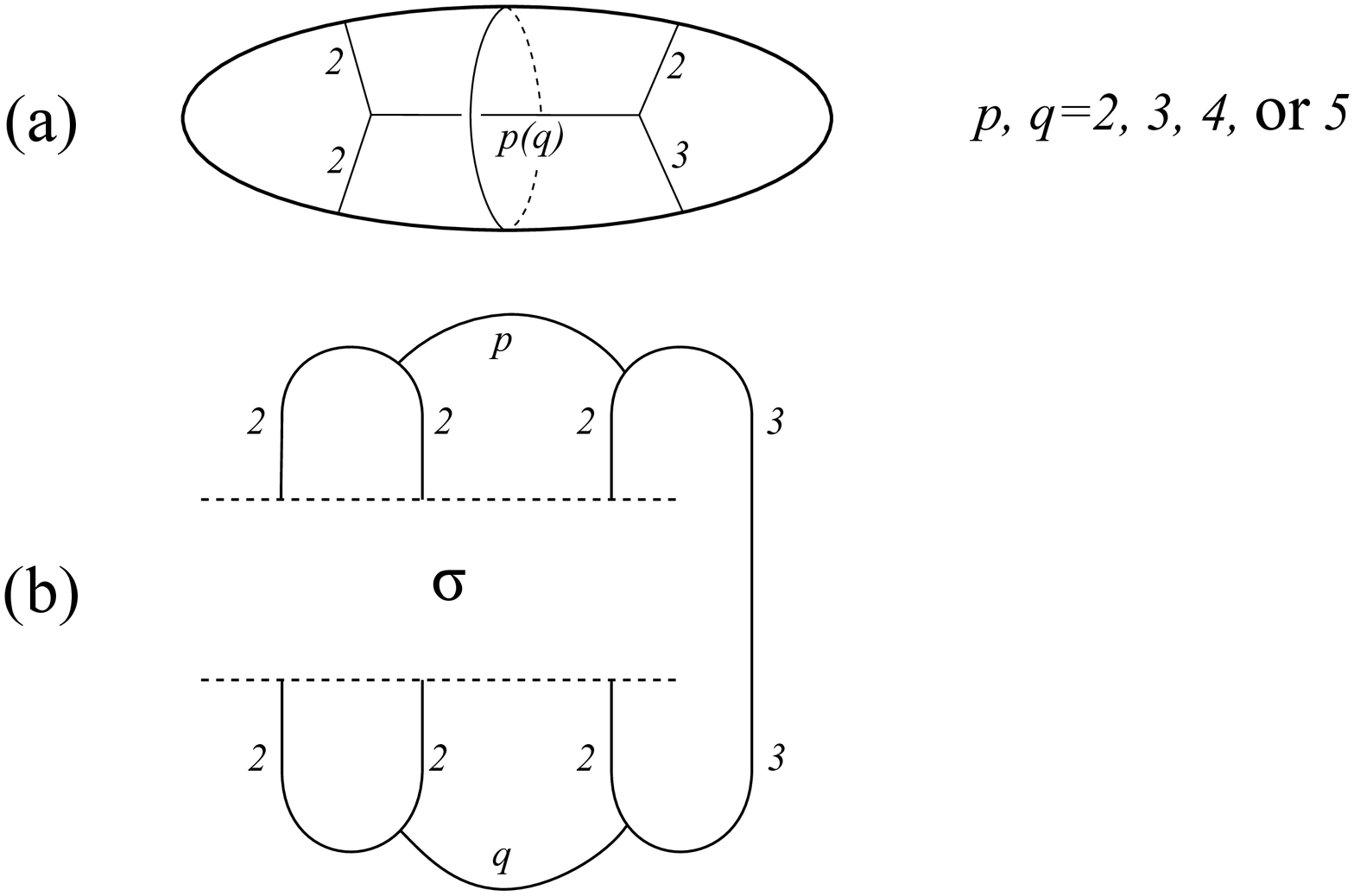}}

Figure 2

\end{center}

The unknotted embedding $ \Sigma_g\to S^3$ provides a Heegaard
splitting $S^3=H_1\cup_{ \Sigma_g} H_2$ such that each handlebody
$H_i$ is invariant under the extendable $G$-action on $(S^3,
\Sigma_g)$. Since $|G|=12(g-1)$, each handlebody orbifold $H_i/G$
has the underlying space $B^3$ and singular set as indicated in
Figure 2 (a). Then the quotient-orbifold $S^3/G$ is $S^3$, the
singular set is a 2-bridge link with the two standard unknotting
tunnels of branching orders $p$ and $q$ added, with $p,q \le 5$, as
indicated in Figure 2(b); three of the strands of the 2-bridge link
have branching order 2, the remaining one branching order 3, where
$\sigma$ is a braid on 3 strands, and these are the orbifolds
$O(\sigma;p,q)$, see \cite{Zi3} for details.

All spherical 3-orbifolds $S^3/G$ with underlying space
$|S^3/G|=S^3$ are listed in Tables 6, 7, 8  of Dunbar's paper
\cite{Du1} (pages 89-93): Each singular set is a graph with vertex
valency at most 3, and each edge is labeled by an integer indicating
the singular index of the edge, with the convention that each
unlabeled edge has index 2. Each small  box encoded by an integer
$k$ indicates two parallel arcs with $k$-half twists in the box, and
each small box encoded by two integers $m,n$ indicates a rational
tangle given by $(m,n)$ with a ``strut" connecting the two arcs of
the tangle and labeled by the largest common divisor of $m$ and $n$.

Now we use  Dunbar's list   \cite{Du1} to see which of the orbifolds
in Figure 2 are spherical. There are two cases.

(1) If  not both $p$ and $q$ are equal to 2 then there is a singular
point which is not dihedral, so it is of type $A_4$ (tetrahedral),
$S_4$ (octahedral) or $A_5$ (icosahedral). In Dunbar's list, these
are only the non-fibered orbifolds on page 93 of \cite{Du1}. By
further checking which graph can meet a 2-sphere $S^2$ with four
singular points of indices $(2,2,2,3)$ so that each side of $S^2$ is
a handlebody orbifold,  we have only nine graphs left which are
listed in  Figure 3. From \cite{Du2}, we also know the fundamental
groups of these nine orbifolds, and  we indicate the groups and
their orders under each graph. Here $O$ denotes the
orientation-preserving symmetry group of the octahedron, and $J$ the
orientation-preserving symmetry group of the icosahedron; all the
symbols are from \cite{Du2}. %(See also p. 269 of \cite{Zi4} for the
%orders of the spherical orbifolds of tetrahedral type; note that all
%orbifolds in Figure 3 except the last one are either of tetrahedral
%type, or admit 2-fold coverings by tetrahedral orbifolds.)

\begin{center}
\scalebox{0.5}{\includegraphics*[0pt,0pt][800pt,400pt]{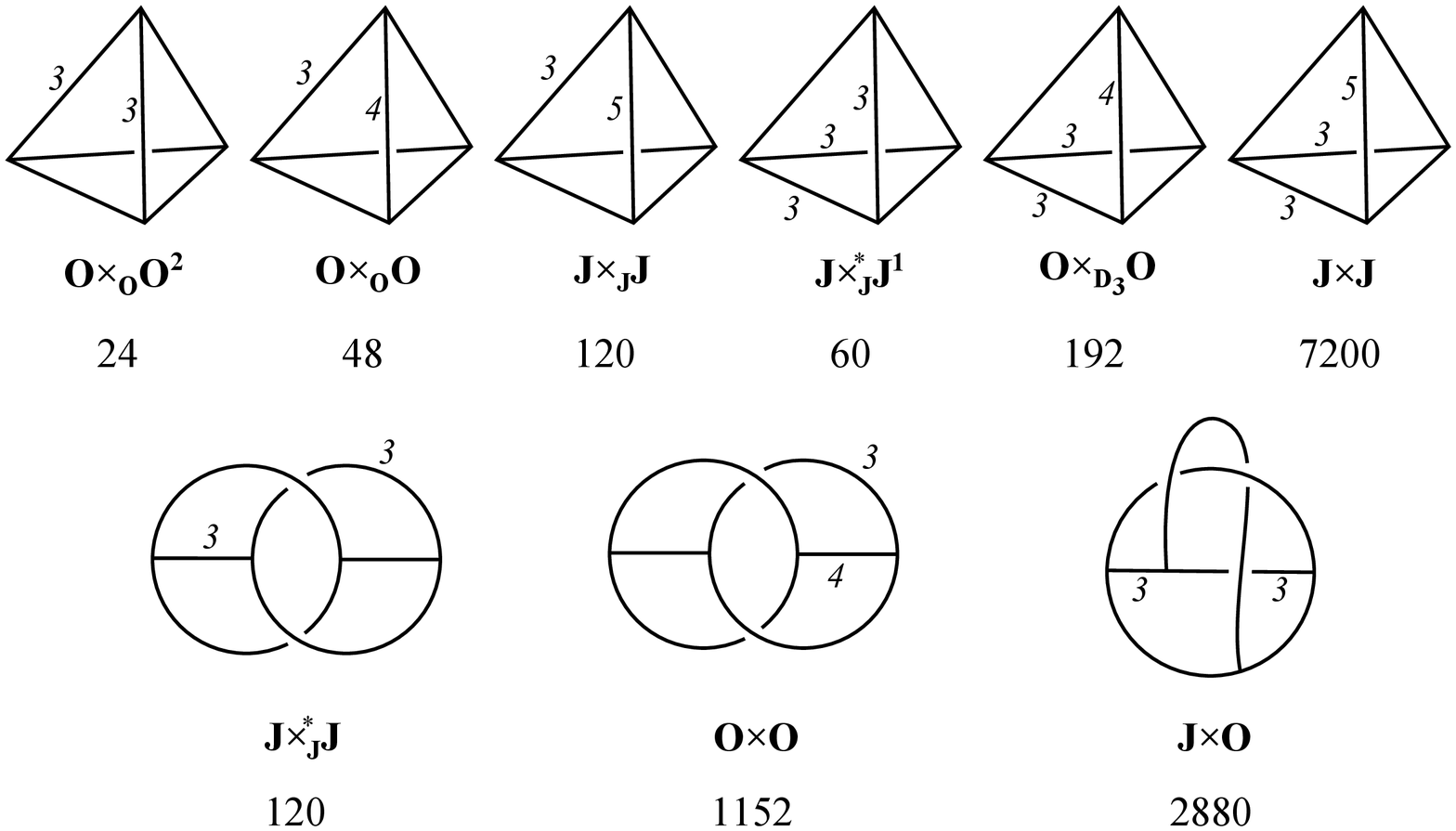}}

Figure 3
\end{center}

From the relation $|G|=12(g-1)$, the cases one finds are $g = 3$, 5,
11, 6, 17, 601, again 11, 97 and 241.

(2) If both $p$ and $q$ are equal to 2, then an easy and more
interesting way to find the cases is as follows. Take any 2-bridge
link $L(\sigma)$ in $S^3$ and associate branching index 3 to each of
its components, obtaining an orbifold $L_3(\sigma)$. It is shown in
\cite{MeZ} that such an orbifold $L_3(\sigma)$ has an
orientation-preserving symmetry group $\Z_2 \times \Z_2$, and the
quotient orbifold is exactly $O(\sigma;2,2)$, for the same 2-bridge
knot defined by $\sigma$. So if $O(\sigma;2,2)$ is spherical, also
$L_3(\sigma)$ is spherical and the singular set is just a 2-bridge
link now. The spherical orbifolds whose singular set is just a link
are listed on pages 89-92 of \cite{Du1}. When we restrict to the two
bridge links such that each component has index 3, then only five
links are left which are listed in Figure 4, and all of them are
very simple torus links.

\begin{center}
\scalebox{0.5}{\includegraphics*[0pt,0pt][800pt,110pt]{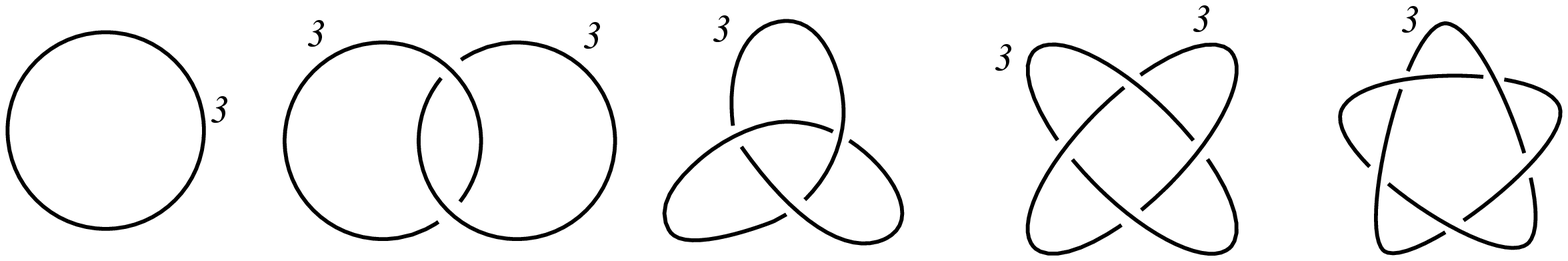}}

Figure 4
\end{center}

Let $\tilde L_3(\sigma)$ be the 3-fold cyclic branched covering of
$L_3(\sigma)$ so that $\tilde L_3(\sigma)$ has no singularity. Then
we have the orbifold covering $$\tilde L_3(\sigma)\to L_3(\sigma)\to
O(\sigma;2,2)$$ of degree 12.

For all $L_3(\sigma)$ in Figure 4, $\tilde L_3(\sigma)$ are
well-known spherical Seifert fiber spaces: The first is $S^3$, the
second the lens space space $L(3,1)$, the third is the quaternion
manifold, the fourth is the 3-manifold with binary tetrahedral
fundamental group, and the fifth is the Poincare homology 3-sphere
(indeed these five 3-manifolds are exactly the $k$-fold cyclic
branched covering over the trefoil knot, for $k = 1$, 2, 3, 4 and 5;
see \cite{Ro}, pp 304-309 for a discussion of the cyclic branched
coverings of the trefoil). The fundamental groups of these orbifolds
$\tilde L_3(\sigma)$  have orders 1, 3, 8, 24 and 120, so for the
corresponding orbifolds  $O(\sigma;2,2)$ one has orders 12, 36, 96,
288 and 1440;  applying again $|G|=12(g-1)$ we obtain the genera 2,
4, 9, 25 and  121.
%As an example, the 3-fold branched covering of the torus link of
%type (2,4) is the Seifert fiber space whose fundamental group is the
%binary tetrahedral group of order 24; this is also the 4-fold
%branched covering of the trefoil or torus knot of type (2,3), by
%considering the quotient of the 4-fold symmetry of the torus link
%(2,4) and exchanging the roles of the branching orders 4 and 3 . So
%$G$ has order 298 and g = 25  in this case, and the other cases are
%similar.
\end{proof}

%\noindent\textbf{Remark on Theorem \ref{OE} and Theorem
%\ref{realzation of 12(g-1)}}

%It is natural to ask the following question,  in particular  in the
%light of  Theorem \ref{OE} and Theorem \ref{realzation of 12(g-1)}.

We end this section by a lemma which should be useful to study
Question 1, and has also some applications in Section 4.

\begin{lemma}\label{connected} Suppose a  group $G$ acts on $(M, F)$
where  $M$ is a  3-manifold and $F\subset M$ a surface, so we have
the diagrams:
$$\xymatrix{
  F \ar[d]_{p} \ar[r]^{i} & M \ar[d]^{p}
  & \pi_1(F) \ar[d]_{p_*} \ar[r]^{i_*} & \pi_1(M) \ar[d]^{p_*} \\
  F/G  \ar[r]^{\hat{i}} & M/G
  & \pi_1(F/G)  \ar[r]^{\hat{i}_*} & \pi_1(M/G)
  }
$$
Suppose  $F/G$ is connected. Then $F$ is connected if
$$\hat{i}_*(\pi_1(F/G))\cdot p_*(\pi_1(M))=\pi_1(M/G).$$
\end{lemma}

\begin{proof} Suppose  $F$ is not connected.  Let $F_1\subseteq F$ be a
component of $F$ and $G_1$ its stabilizer in $G$, that is $G_1=\{
h\in G | h(F_1)=F_1\}$. Then $F_1/G_1=F/G$. Now
$|\pi_1(M/G):p_*(\pi_1(M))|=|G|$, and
%we will compute$|\hat{i}_*(\pi_1(F/G))\cdot p_*(\pi_1(M)):p_*(\pi_1(M))|$
\begin{eqnarray*}
&&|\hat{i}_*(\pi_1(F/G))\cdot p_*(\pi_1(M)): p_*(\pi_1(M))| \\
&= &|\hat{i}_*(\pi_1(F/G))\cdot p_*(\pi_1(M))/ p_*(\pi_1(M))|\\
&= &|\hat{i}_*(\pi_1(F/G))/\hat{i}_*(\pi_1(F/G))\cap p_*(\pi_1(M))|\\
&\leq &|\hat{i}_*(\pi_1(F/G)):\hat{i}_*p_*(\pi_1(F_1))|\\
&= &|\pi_1(F/G))/ker\hat{i}_*:p_*(\pi_1(F_1))\cdot
ker\hat{i}_*/ker\hat{i}_*|\\  &= &|\pi_1(F_1/G_1)):p_*(\pi_1(F_1))\cdot
ker\hat{i}_*|\\
&\leq &|\pi_1(F_1/G_1)):p_*(\pi_1(F_1))|\\
&= &|G_1|< |G|.
\end{eqnarray*}
Hence $\hat{i}_*(\pi_1(F/G))\cdot p_*(\pi_1(M))\subsetneqq
\pi_1(M/G)$.
\end{proof}

\noindent\textbf{Remark} In Lemma \ref{connected}, if $M$ is $S^3$,
then $F$ is connected if $\hat{i}_*$ is surjective.

\section{Maximum orders of extendable abelian and cyclic groups }

To get the maximum orders $AE_g$ and $CE_g$,  we first need some
information about actions of abelian groups and cyclic groups on
handlebodies which are contained in the following Theorem
\ref{abelian1} and Theorem \ref{cyclic2}.  Some facts in Theorem
\ref{abelian1} and Theorem \ref{cyclic2} have been either explicitly
or implicitly stated, with or without proofs, in \cite{MMZ},
\cite{MeZ}. For our later applications, we reorganize them into our
statement and provide  proofs.

 We define two actions of a finite
group $G$ to be {\it equivalent} if the corresponding groups of
homeomorphisms of $V_g$ are conjugate (i.e., allowing isomorphisms
of $G$).

\begin{theorem}\label{abelian1}
The largest order of a finite abelian group $G$ acting on the
handlebody $V_g$ of genus $g \ge 2$ is $2(g+1)$ if $g \ne 5$, and 16
if $g=5$. The  groups $G$ which realize the maximum orders are $\Z_2
\times \Z_{g+1}$ if $g \ne 5$,  $(\Z_2)^4$ if $g=5$, and in addition
$(\Z_2)^3$ if $g=3$. Moreover,

 (i) there is one equivalence
class for each of the groups $\Z_2 \times \Z_{g+1}$ and $(\Z_2)^4$
whereas there are three equivalence classes for the group $(\Z_2)^3$
acting on $V_3$;

(ii) no abelian group of order larger than 12 acts on $V_5$ except
$(\Z_2)^4$.
\end{theorem}

\begin{proof}

Suppose that $G$ is an abelian group as in Theorem \ref{abelian1}.
Then $g \ge 2$ implies $-\chi > 0$ by (2.2). Also, every vertex
group of $\G$ is either cyclic or isomorphic to the dihedral group
$(\Z_2)^2$ of order 4, since these are the finite abelian subgroups
of $SO(3)$; then the group of every edge which is not a loop is
either trivial or $\Z_2$, and in the second case the two adjacent
vertex group are $(\Z_2)^2$.

\medskip

Note that, if $|G| \ge 2g-1$, then  $-\chi = (g-1)/|G| \le
(g-1)/(2g-1) < 1/2$. We will assume that $-\chi < 1/2$ in the
following and discuss all possibilities for $\G$ and $G$. The
discussion is divided into two cases:

\medskip

(I) Suppose first that $\G$ has no vertex group $(\Z_2)^2$; then all
vertex groups are cyclic. Let $$E=\{\text{edge } e\in \Gamma | G_e
\text{ is non-trivial}\};$$ then all edges in $E$ must be loops. Let
$\Gamma_0=\Gamma-E$; if we view $\Gamma_0$ as a usual graph, we
denote it by $|\Gamma_0|$. It is easy to see
$-\frac{1}{2}<\chi(\Gamma)\leq\chi(\Gamma_0)\leq\chi(|\Gamma_0|)$.
However $\chi(|\Gamma_0|)$ is an integer, so $\chi(|\Gamma_0|)=1$ or
$0$. Suppose $\Gamma_0$ has $k$ non-trivial vertices; then
$\chi(\Gamma_0)\leq\chi(|\Gamma_0|)-\frac{k}{2}$. So if
$\chi(|\Gamma_0|)=0$, it is easy to see $\chi(\Gamma)\notin
(-\frac{1}{2},0)$. So we must have $\chi(|\Gamma_0|)=1$, which means
$|\Gamma_0|$ must be a tree. Notice that every end (degree-one
vertex) of $\Gamma_0$ must be non-trivial, hence the same reason as
above shows that $\Gamma_0$ is equal to some
$\Gamma(\Z_{n_1},1,\Z_{n_2})$ (consisting of one edge with trivial
edge group, and two vertices). Furthermore, $(n_1,n_2)\in
\{(3,5),(3,4),(3,3),(2,n)\}$, $n\geq3$, and $\Gamma$ must equal to
$\Gamma_0$. So the only possibilities for $\Gamma$ are the graphs of
groups $\Gamma(\Z_2,1,\Z_n)$, with $-\chi = (n-2)/2n$, and
$\Gamma(\Z_3,1,\Z_n)$, for $n = 3$, 4 or 5, with $-\chi = 1/3$, 5/12
or 7/15. The only finite abelian groups onto which the free product
$\pi_1\Gamma(\Z_2,1,\Z_n) \cong  \Z_2 * \Z_n$ surjects with
torsionfree kernel are the groups $\Z_2 \times \Z_n$ and, if $n$ is
even, $\Z_n$. In the first case we have $g = n-1$ and $|G| = 2n =
2(g+1)$, so for all genera $g$ the order $2(g+1)$ of the Theorem is
achieved for the group $\Z_2 \times \Z_{g+1}$. In the three other
cases, the possibilities for $G$ are the groups $\Z_3$ and
$(\Z_3)^2$, $\Z_{12}$ or $\Z_{15}$, and in each of these cases one
has $|G| < 2(g+1)$.

\medskip

(II) Suppose now that $\G$ has some vertex group  $(\Z_2)^2$. Let
$$E=\{\text{edge } e\in \Gamma | G_e \text{ is non-trivial, and the
ends of $e$ are both cyclic groups}\};$$ also all edges in $E$ must
be loops. Setting $\Gamma_0=\Gamma-E$, every non-trivial edge in
$\Gamma_0$ must have both ends $(\Z_2)^2$. We now also have
$-\frac{1}{2}<\chi(\Gamma)\leq\chi(\Gamma_0)\leq\chi(|\Gamma_0|)$.
We explain the last inequality: suppose $\Gamma_0$ has $l$ vertices
of type $(\Z_2)^2$; then $\Gamma_0$ has no more than $\frac{3l}{2}$
non-trivial edges, so
$\chi(\Gamma_0)\leq\chi(|\Gamma_0|)-\frac{3}{4}l+\frac{1}{2}\cdot
\frac{3l}{2}=\chi(|\Gamma_0|)$. We also have $\chi(|\Gamma_0|)=1$ or
$0$. Now for a cyclic end, $\chi(\Gamma_0)$ must decrease at least
by $-\frac{1}{2}$ compared with $\chi(|\Gamma_0|)$. For a $(\Z_2)^2$
end, $\Gamma_0$ has no more than $\frac{3l-2}{2}$ non-trivial edges,
and this time
$\chi(\Gamma_0)\leq\chi(|\Gamma_0|)-\frac{3}{4}l+\frac{1}{2}\cdot
\frac{3l-2}{2}=\chi(|\Gamma_0|)-\frac{1}{2}$, so it also decrease at
least by $-\frac{1}{2}$. This argument show that $\Gamma_0$ has at
most $2$ ends.

(1) If $\chi(|\Gamma_0|)=0$, $\Gamma_0$ can even have no ends at
all, so it is a loop divided by some $(\Z_2)^2$ vertices. The only
possibility for $\G$ is  a graph of groups with exactly one vertex
and one edge (a loop), with $\chi = -1/4$; however, since the
HNN-generator of $\pi_1\G$ corresponding to the loop has to
conjugate a subgroup $\Z_2$ of the vertex group $(\Z_2)^2$ into a
different subgroup $\Z_2$ (see \cite{MMZ} or \cite{Zi3}), its
fundamental group does not surject onto an abelian group and this
case does not occur.

(2) If $\chi(|\Gamma_0|)=1$, $\Gamma_0$ is a segment. The segment
may have inner vertices, but every inner vertex must be $(\Z_2)^2$.
For every such inner vertex, $\Gamma_0$ has no more than
$\frac{3l-1}{2}$ non-trivial edges, and this time
$\chi(\Gamma_0)\leq\chi(|\Gamma_0|)-\frac{3}{4}l+\frac{1}{2}\cdot
\frac{3l-1}{2}=\chi(|\Gamma_0|)-\frac{1}{4}$. So there is at most
one inner vertex. If there is no inner vertex, it is easy to see
that $\G$ is equal to $\Gamma((\Z_2)^2,1,\Z_2)$, with $-\chi = 1/4$,
or to $\Gamma((\Z_2)^2,1,\Z_3)$, with $-\chi = 5/12$; the
possibilities for $G$ are the groups $(\Z_2)^2$ and $(\Z_2)^3$ in
the first case, and $\Z_2 \times \Z_6$ in the second one, and only
for the group $(\Z_2)^3$, with $g = 3$, the bound $|G| = 2(g+1)$ of
the Theorem is obtained. The last case is the graph of groups
$\Gamma((\Z_2)^2, \Z_2, (\Z_2)^2, \Z_2, (\Z_2)^2)$, with two edges
and three vertices and $-\chi = 1/4$, whose fundamental group
surjects onto $(\Z_2)^n$ for $n = 2$, 3 and 4; the group $(\Z_2)^4$
realizes the maximum order 16 for $g = 5$ of the Theorem, whereas
the group $(\Z_2)^3$ realizes again the maximum order $8 = 2(g+1)$
for $g = 3$

There is only one finite-injective surjective map from $\Z_2 \times
\Z_{g+1}$ or from $\Gamma((\Z_2)^2, \Z_2, (\Z_2)^2, \Z_2, (\Z_2)^2)$
to an abelian group. There are two finite-injective surjective map
from $\Gamma((\Z_2)^2, \Z_2, (\Z_2)^2, \Z_2, (\Z_2)^2)$ to
$(\Z_2)^3$  (either all three vertex groups $(\Z_2)^2$ are mapped to
different subgroups of $(\Z_2)^3$, or two vertex groups are mapped
to the same subgroup), and one finite-injective surjective map from
$\Gamma((\Z_2)^2,1,\Z_2)$ to $(\Z_2)^3$. These show the result in
(i). For (ii), an abelian group of order $13$, $14$ or $15$ must be
a cyclic group, and the above orbifolds can not finite-injectively
surject to such a cyclic group when $g=5$.

This completes the proof of Theorem \ref{abelian1}.
\end{proof}

For cyclic groups, the proof of Theorem \ref{abelian1} implies also
the following:

\begin{theorem}\label{cyclic2} Let  $G$ be a finite cyclic group acting
on a handlebody  of genus $g\ge 2$. If $G$ has order at least $2g-2$
then $G$ is one the following groups:
\setlength{\parskip}{\smallskipamount}

(1) $\Z_{2g+2}$ if $g$ is even, associated to a surjection $\Z_2
* \Z_{g+1} \to \Z_{2(g+1)}$;

(2) $\Z_{2g}$ for all $g$, associated to a surjection $\Z_2 *
\Z_{2g} \to \Z_{2g}$ and, for $g=6$, also to $\Z_3 * \Z_4 \to
\Z_{12}$;

(3) $\Z_{2g-1}$ for $g=2$ and 8, associated to surjections $\Z_3 *
\Z_3 \to \Z_3$ and $\Z_3 * \Z_5 \to \Z_{15}$;

(4) $\Z_{2g-2}$ for all $g$;  for each $g$ the graphs of groups are
$\Gamma(\Z_2,1,\Z_n)$ with an additional loop with edge group $\Z_n$
attached to the vertex of type $\Z_n$, for each $n \ge 1$ which
divides $2g-2$;  in addition, for $g=3$ and 2 there are actions
associated to  $\Z_4 *\Z_4 \to \Z_4$ and $\Z_2 * \Z_2 * \Z_2 \to
\Z_2$.
\end{theorem}

\begin{proof}  The cases $|G| \ge 2g-1$ are covered by the proof of
Theorem \ref{abelian1}. For $|G| = 2g-2$ one has to consider in
addition graphs of groups with $-\chi = 1/2$. If $\Gamma_0$ has at
least three non-trivial vertices, it gives the case $\Z_2 * \Z_2 *
\Z_2 \to \Z_2$. Or $\Gamma_0=\Gamma(\Z_{n_1},1,\Z_{n_2})$ and there
are two more cases: $(n_1,n_2)$=$(4,4)$, or $\Gamma$ is obtained
from $\Gamma_0$ by adding a non-trivial loop as stated in (4).
\end{proof}

\begin{theorem}\label{abelian} $AE_g=2g+2.$
%Suppose $G$ is a finite abelian group which can
%orientation-preserving act on an orientable closed surface
%$\Sigma_g$, $g>1$. If this action can extend to a $G$-action on
%$S^3$ for some embedding $e:  \Sigma_g\hookrightarrow S^3$, then the
%maximum order of $G$ is $2g+2$.
\end{theorem}

\begin{proof}
Suppose  an abelian group $G$ acts on $ \Sigma_g$ which is
extendable over $S^3$ for some embedding $e: \Sigma_g\hookrightarrow
S^3$. Then the action of $G$ extends to each 3-manifold of
$S^3\setminus \Sigma_g$. According to \cite[Theorem 2]{RZ}, for each
abelian group $G$, the $G$-action on $\Sigma_g$ extends to a compact
$3$-manifold $M$ with $\partial M= \Sigma_g$ if and only if the
$G$-action on $\Sigma_g$ extends to a handlebody $V_g$ with
$\partial V_g= \Sigma_g$. Therefore we have $AE_g\le AH_g$.

It is a general fact of Smith fixed point theory that the finite
$2$-group $(\Z_2)^{(n+1)}$ does not act orientation-preservingly on
a $mod~2$ homology $n$-sphere, see \cite{Sm}. So a $(\Z_2)^4$-action
on $F_5$ does not extend to an action on $S^3$. (Alternatively, by
the confirmation of the geometrization conjecture one can use the
fact that every finite group acting orientation-preservingly on
$S^3$ can be conjugated into $SO(4)$, and $SO(4)$ contains no
subgroup isomorphic to $(\Z_2)^{4}$.) Now applying Theorem
\ref{abelian1}, we have indeed $AE_g\le 2(g+1)$ for each $g>1$.

By Example 4.1, for every $g>1$ there is an abelian group $G\cong
\mathbb{Z}_2\times\mathbb{Z}_{g+1}$ which acts on $\Sigma_g$,  and
this action extends to a $G$-action on $S^3$ for the unknotted
embedding of $\Sigma_g\subset S^3$. Hence $AE^o_g\ge 2(g+1)$.

Then

$$2(g+1)\le AE^o_g\le AE_g\le 2(g+1).$$

%(Indeed  we can show that an element of $(\Z_2)^4$ action on $V_5$
%which gives the obstruction of extending to $S^3$, without using
%Smith theory, See Example 5.1.)

So we have $AE_g=AE^o_g= 2(g+1).$ and  Theorem \ref{abelian} is
proved.\end{proof}

The following fact proved without using Smith theory is  of
independent interest.

\begin{lemma}
Some order 2 element of the $(\Z_2)^4$ action on $V_5$ is not
extendable.
\end{lemma}

\begin{center}
\scalebox{0.35}{\includegraphics*[0pt,0pt][515pt,515pt]{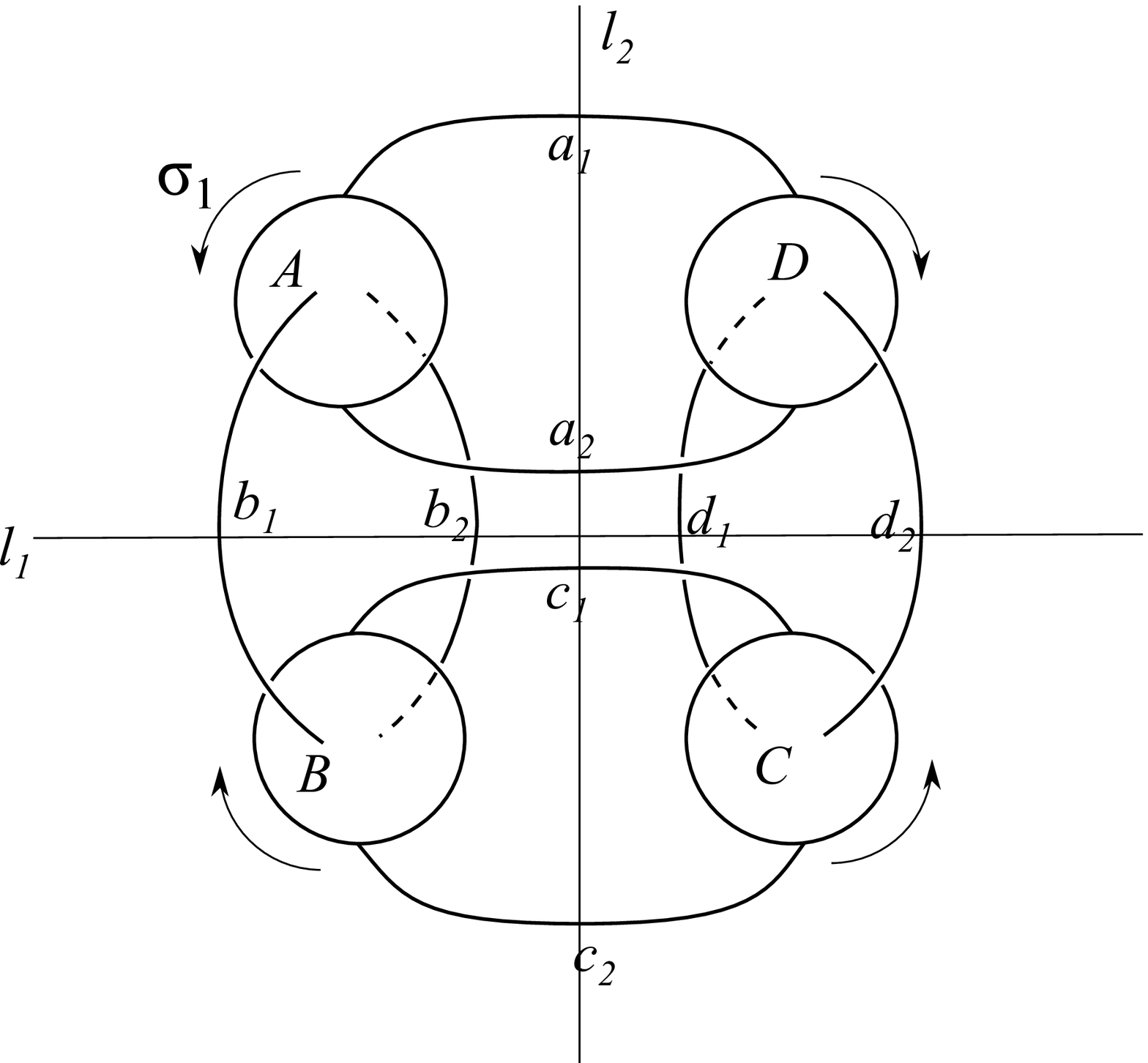}}

Figure 5
\end{center}

\begin{proof}  We first describe a geometric model of the action of
$(\Z_2)^4$  on $V_5$.

As in Figure 5, we view $V_5$ as four $3$-balls $\{A, B, C, D\}$
with $8$ handles $\{a_1, a_2, b_1, b_2, c_1, c_2, d_1, d_2\}$
attached; here, for simplicity, we draw each handle as an arc, and
the four attaching disks on each ball are at front, back, top and
bottom respectively.

\begin{center}
\scalebox{0.35}{\includegraphics*[0pt,0pt][950pt,410pt]{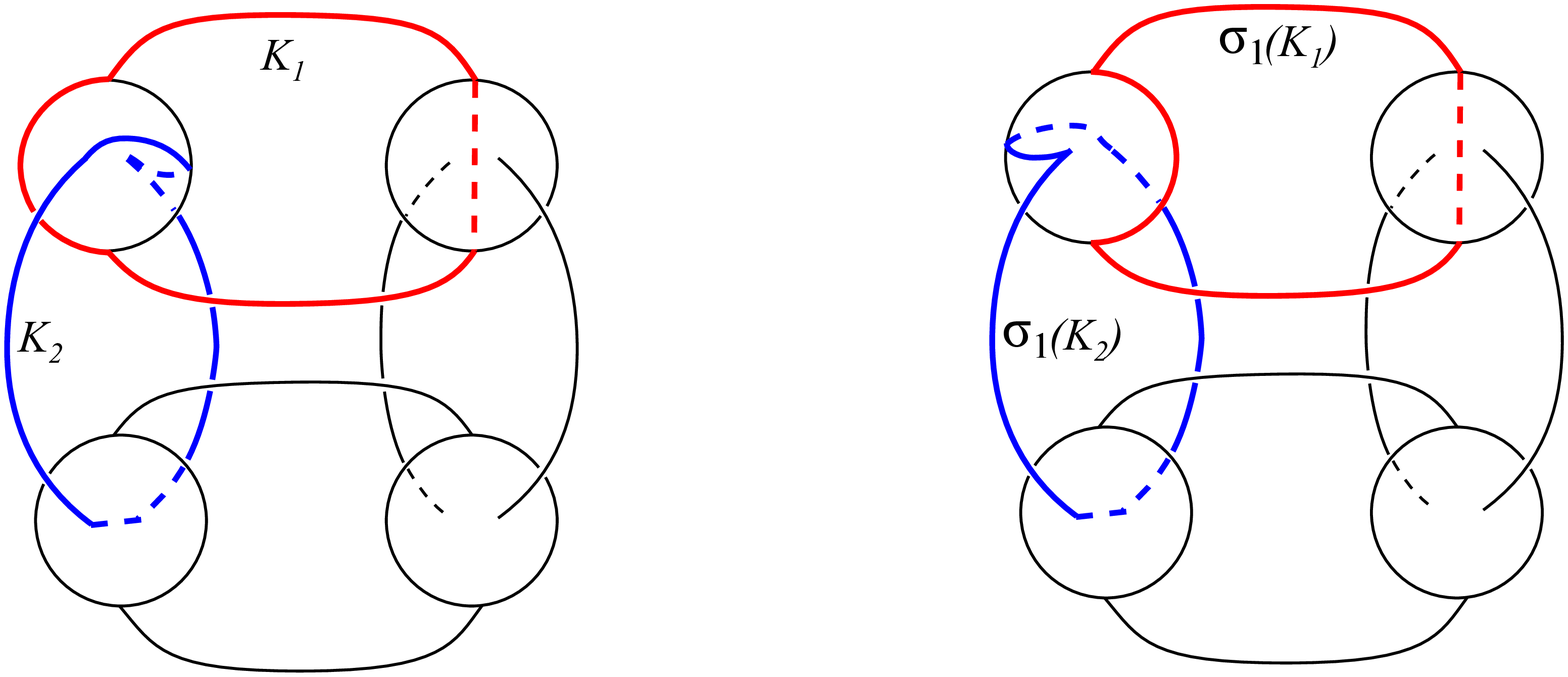}}

Figure 6
\end{center}

We will describe the four generators $\{\sigma_1, \sigma_2, \rho_1,
\rho_2\}$ of the action of $ G\cong (\Z_2)^4$: on each ball
$\sigma_1$ is a $\pi$-rotation about the axis through the front and
back points along the arrows shown as in Figure 6, which
interchanges $a_1$ and $a_2$ (resp. $c_1$ and $c_2$) but keeps each
$b_{i}$ and $d_{i}$, $j=1,2$. Similarly $\sigma_2\in G$ is a
$\pi$-rotation about the axis through the top and bottom points on
each ball, which interchanges $b_1$ and $b_2$, (resp. $d_1$ and
$d_2$), but keeps each $a_{i}$ and $c_{i}$, $j=1,2$. The generator
$\rho_i$ is a $\pi$-rotation of the whole handlebody about the axis
$l_i$, $i=1,2$. One can check that
$\{\sigma_1,\sigma_2,\rho_1,\rho_2\}$ generates the abelian group
$G$. Note that $\rho_1$ and $\rho_2$ are extendable.

For any embedding $V_5 \hookrightarrow S^3$, we construct a link
$\{K_1,K_2\}$ in the handlebody $V_5$ as in Figure 6 where the image
$\sigma_1\{K_1,K_2\}$ is shown on the right hand side. The linking
numbers locally differ by $1$ in the upper-left ball. So $\sigma_1$
is not extendable.
\end{proof}

\begin{theorem}\label{cyclic} $CE_g =  2g+2$ if $g$ is even, and $2g-2$ if
$g$ is odd.
% Suppose $G$ is an
%finite cyclic group which can orientation-preserving act on an
%orientable closed surface $\Sigma_g$, $g>1$. If this action can
%extend to a $G$-action on $S^3$ for some embedding
%$\Sigma_g\hookrightarrow S^3$, then the maximum order of $G$ is
%$2g+2$ when $g$ is even, and $2g-2$ when $g$ is odd.
\end{theorem}

\begin{proof}
In Example 4.1 we shall describe  a cyclic group action of order
$2g+2$ on $\Sigma_g$ which extends over $S^3$ for each even  $g>1$,
and in Example 4.2 a cyclic group action of order $2g-2$ on
$\Sigma_g$ which extends over $S^3$ for each odd  $g>1$. Hence
$$CE_g\ge 2g+2 \text{ for even $g>1$},\,\,\, CE_g\ge 2g-2 \text{ for  odd
$g>1$.}   \qquad (3.3)$$

Suppose the $G$ action on $\Sigma_g$ is extendable. Still applying
\cite[Theorem 2]{RZ}, we have that the $G$ action on $\Sigma_g$
extends to $(V_g, \partial V_g=\Sigma_g)$.

 By Theorem \ref{cyclic2} (1) for each even $g$ we have
$AH_g=2g+2$ . By Theorem \ref{cyclic2} (2) (3) for each odd $g$ a
cyclic group $G$ of order $|G|>2g-2$ acting on $V_g$ must be
$\Z_{2g}$, is associated to surjection $\Z_2 * \Z_{2g} \to \Z_{2g}$.

\noindent \textbf{Claim:} {\it The $\Z_{2g}$ action on $\partial
V_g=\Sigma_g$ which is the restriction of the $\Z_{2g}$ acts on
$V_g$ is not extendable.}

With this Claim we have
$$CH_g\le 2g+2 \text{ for even $g>1$},\,\,\, CH_g\le 2g-2 \text{ for  odd
$g>1$.}\qquad (3.4)$$

Combining (3.3) and (3.4), Theorem \ref{cyclic} is proved.

\noindent \textbf{Proof of the Claim.} We must have a close look on
the $\Z_{2g}$ action on $V_g$. By the discussion made in \cite{MMZ},
\cite{Zi3}, the handlebody orbifold $X=V_g/\Z_{2g}$ must consist of
two 3-balls with singular arcs of indices 2 and $2g$, respectively,
connected by a regular 1-handle as shown in the left hand side of
Figure 7. Now the pre-image of the  $3$-ball with singular arc of
index ${2g}$ is just an ordinary 3-ball $B^3$ in  $V_g$, the
$\mathbb{Z}_{2g}$-action on it is a $\frac{\pi}{g}$-rotation, and
the pre-image of the remaining part of the handlebody orbifold $X$
in $V_g$ consists just of $g$ 1-handles attached to opposite
$\Z_{2g}$-equivariant disks on $B^3$. The right hand side  of Figure
7 is the case of $V_3$.

Hence the $\Z_{2g}$ action on $\Sigma_g=\partial
V_g=S^2_*\cup\{N_1,..., N_g\}$ is obtained from the 2-sphere $S^2_*$
with $2g$ punctures by attaching $g$ tubes  $N_1, ..., N_g$ along
$g$ pairs of opposite punctures.
\begin{center}
\scalebox{0.7}{\includegraphics*[0pt,0pt][435pt,190pt]{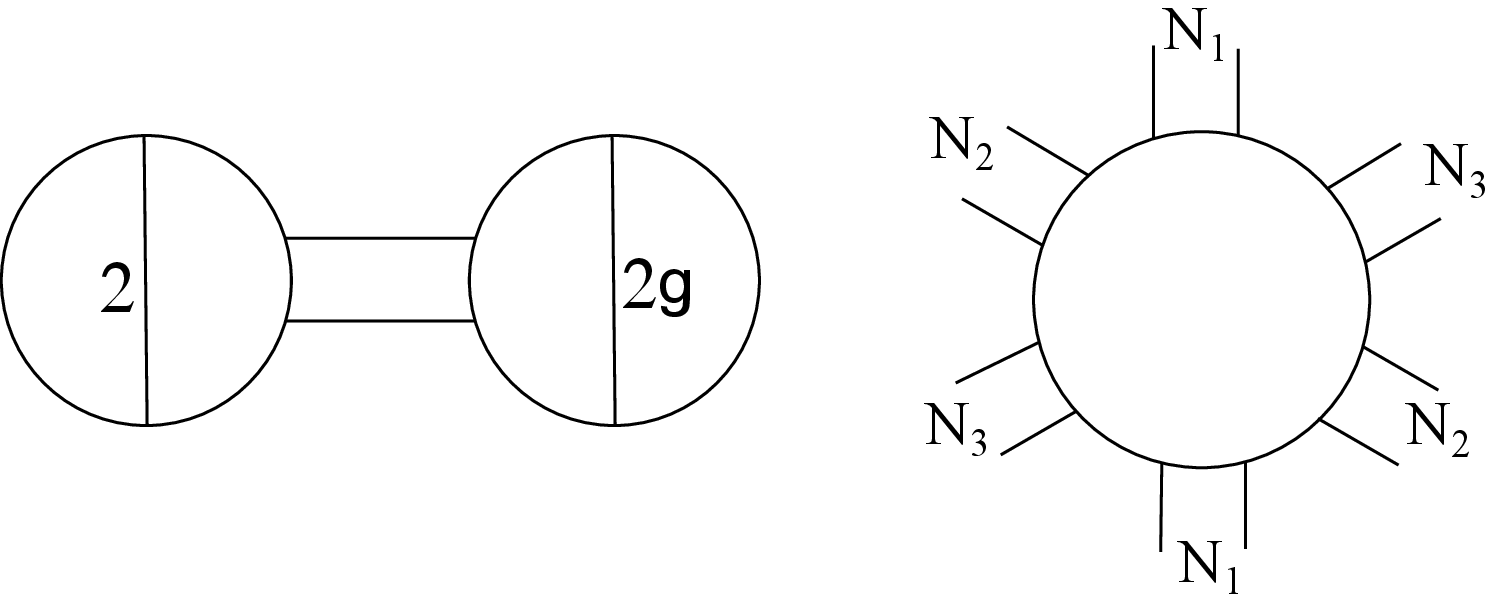}}

Figure 7
\end{center}

Now we are going to give two different proofs that this $\Z_{2g}$
action on $\Sigma_g$ is not extendable.

\noindent \textbf{Proof 1.} The first one invokes Smith theory.

  Suppose the $\Z_{2g}$ action on $\Sigma_g$ extends an
action on $(S^3, \Sigma_g)$ for some embedding $\Sigma_g\to S^3$,
and let $\sigma$ be a generator of this extension. Since
$\sigma|S^2_*$ is a rotation of order $2g$  with two fixed points,
the fixed point set of $\sigma$ is not empty. By Smith theory the
fixed point set of the group $<\sigma>$ acting on $S^3$ must be a
circle $C$.  It follows that the (singular or branching) index of
$C\cap S^2_*$ must be $2g$, therefore the index of the whole circle
$C$ must be $2g$, that is to say the whole $C$ is the fixed point
set $\sigma$. On the other hand, $\sigma^g$ is a $\pi$-rotation on
the tube $N_1$ which has two fixed points $x, y$, therefore $x \in
C$ which implies that $\sigma$ has fixed points on $N_1$. Since
$g>1$, $\sigma$ sends the whole $N_1$ to $N_2$ which gives a
contradiction.

\noindent \textbf{Proof 2.} The second one is elementary and using
linking number only.

Choose an arc $\gamma$ on the boundary of the orbifold $X$, as
showed in Figure 8. The pre-image of $\gamma$ consists of $g$ arcs
$\gamma_i$, $i=1,\ldots, g$ on the surface $\Sigma_g$, equivariant
under the action of $G$. Let $D$ denote the upper  hemisphere of the
$2g$-punctured sphere $S^2_*$ described above. Then the boundary of
$\gamma_i$ divides $\partial D$ into $2g$ arcs denoted by $\alpha_i$
and $\beta_i$ such that the $\frac{\pi}{g}$-rotation maps $\alpha_i$
to $\beta_i$; see the right hand side of Figure 8.

\begin{center}
\scalebox{0.6}{\includegraphics*[0pt,0pt][524pt,166pt]{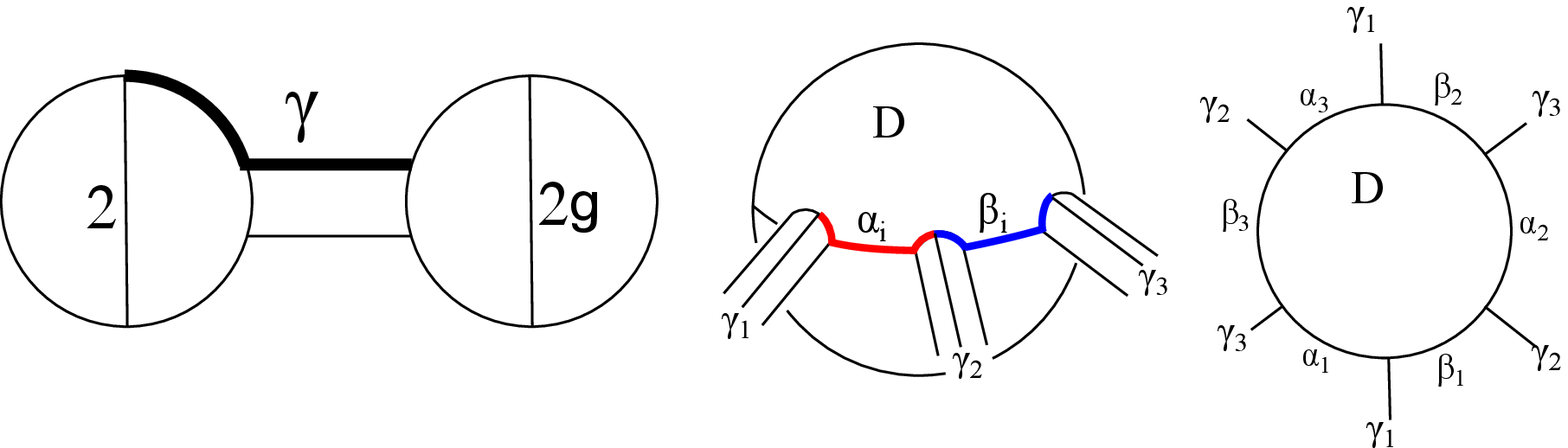}}

Figure 8
\end{center}

Now consider the embedding $\Sigma_g\hookrightarrow S^3$, and let
$$K_1 =  \gamma_i \bigcup \alpha_i,\,\,\,
K_2 =  \gamma_i \bigcup  \beta_i$$ which are knots in $S^3$, see
Figure 9.

\begin{center}
\scalebox{0.6}{\includegraphics*[0pt,0pt][296pt,163pt]{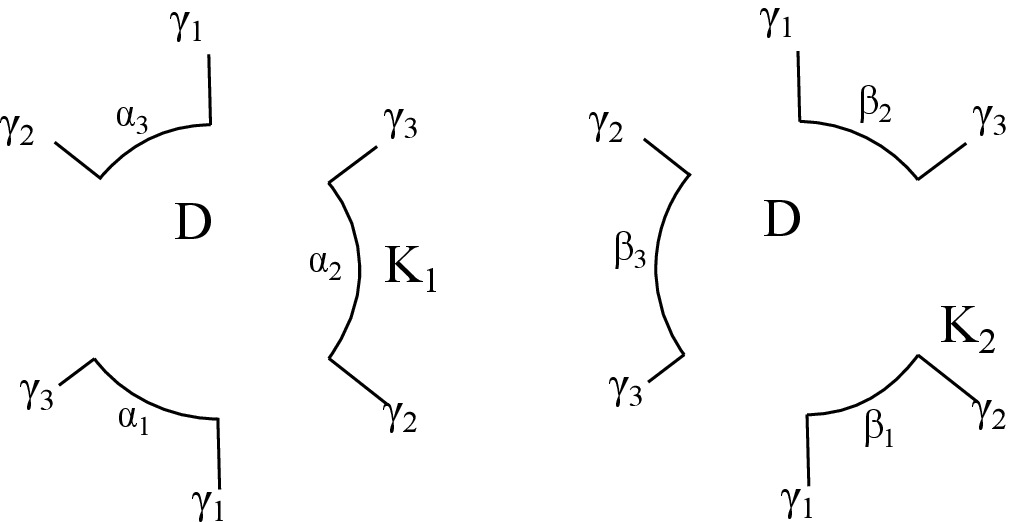}}

Figure 9
\end{center}

Suppose the $G$-action extends to a $G$-action on $S^3$, and let
$\sigma\in G $ be the generater of $G$ such that the restriction of
$\sigma$ to $D$ is the $\frac{\pi}{g}$-rotation. Then we have
$\sigma(K_1)=K_2$. The fixed point set of the periodic map $\sigma$
is not empty (it has a fixed point $x$ on $D$). Denote by $K_0$ the
circle component of the fixed point set of $\sigma$ containing $x$.
%(by
%Smith conjecture, $K_0$ must be an unknot, but we do not need to use
%this).
 Now $\sigma\{K_0,K_1\}=\{K_0,K_2\}$.

Let us compute the mod 2 linking numbers $lk_2(K_0,K_1)$ and
$lk_2(K_0,K_2)$ to reach a contradiction.
%In fact, we only need to
%compute the $mod 2$ linking numbers. This is much easier: we even do
%not need to give $K_i$ an orientation.
Choose a standard embedded $S^2$ in $S^3$, which contains the disk
$D$. Project each knot $K_i$ and $K_0$ to this $S^2$; we use this
projected diagram to compute the linking numbers $lk_2(K_0,K_i)$:
for each crossing, if the arc $K_0$ goes over the arc of $K_i$, then
this crossing contributes $1$ to $lk_2(K_0,K_i)$, and if the arc
$K_0$ goes under the arc of $K_i$ then this crossing contributes $0$
to $lk_2(K_0,K_i)$ (this is just an application of  the general
linking number method, see \cite{Ro} for details).

We notice that the only differences between $lk_2(K_0,K_1) $ and
$lk_2(K_0,K_2)$ are coming from the crossings of $K_0$ with
$\partial D$. There are three cases, shown in Figure 10.

\noindent \textbf{Case 1.} Both ends of the arc in the disk go over
$\partial D$. If the two ends are both over $\alpha_i$ or both over
$\beta_i$, then it contribute $0$ to both $lk_2(K_0,K_1)$ and
$lk_2(K_0,K_2)$. If one end goes over $\alpha_i$ and the other goes
over $\beta_i$ then this arc contributes $1$ to both $lk(K_0,K_1)_2$
and $lk_2(K_0,K_2)$.

\noindent \textbf{Case 2.} Both ends of the arc in the disk go under
$\partial D$. In this case the arc contributes $0$ to both
$lk_2(K_0,K_1)$ and $lk_2(K_0,K_2)$.

\noindent \textbf{Case 3.} One  end of the arc in the disk goes over
$\partial D$, and the other end goes under $\partial D$. In this
case, the arc of $K_0$ must intersect $D$ in its interior, and that
means it goes through the only fixed point in $D$. So there is
exactly one such arc. If the ``over end" is over $\alpha_i$, then
this arc contributes $1$ to $lk_2(K_0,K_1)$ but $0$ to
$lk_2(K_0,K_2)$. If the ``over end" is over $\beta_i$ then this arc
contributes $0$ to $lk_2(K_0,K_1)$ but $1$ to $lk_2(K_0,K_2)$.
Anyway, it is not the same for the two linking numbers.

\begin{center}
\scalebox{0.7}{\includegraphics*[0pt,0pt][375pt,141pt]{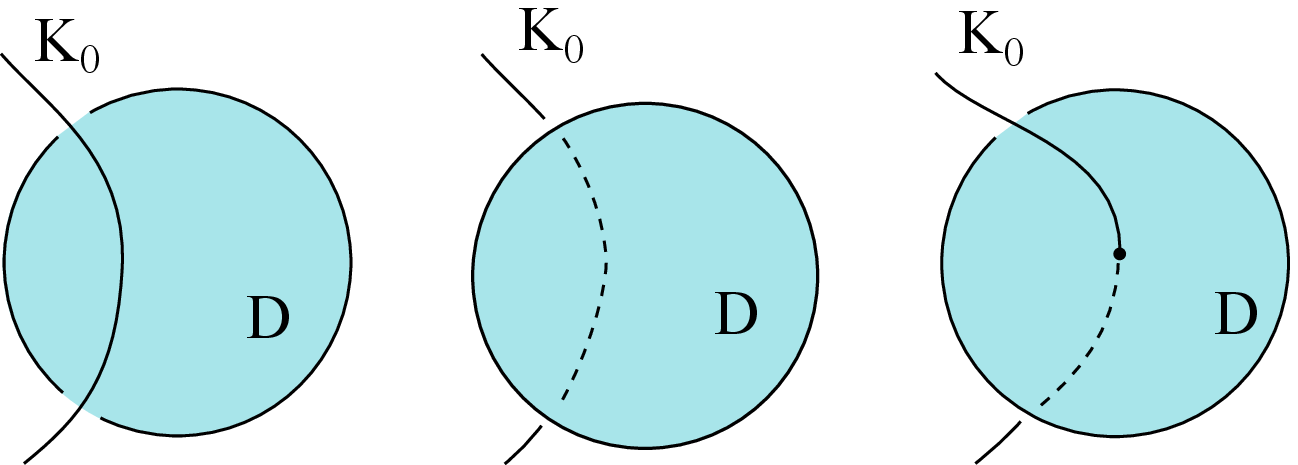}}

Figure 10
\end{center}

From the above, we must have that
$$lk_2(K_0,K_1)\neq lk_2(K_0,K_2).$$
Since $\sigma$ is an automorphism of $S^3$ which send $\{K_0,K_1\}$
to $\{K_0,K_2\}$, this is a contradiction. So for $|G|=2g$ the
$G$-action on $\Sigma_g$ does not extend to a $G$-action on $S^3$.
\end{proof}

%\noindent \textbf{Remark:} This argument will also work if we
%replace the $S^3$ in Theorem 3.1 by any $\mathbb{Z}_2$-homological
%sphere.

%\section{A brief recall on handlebody orbifold theory}

\section{Intuitive view of large symmetries of $(S^3,\Sigma_g)$}

In this section we will present some extendable group actions on
surfaces with ``large" symmetry. Also if the existence of such
examples often can be derived from the powerful orbifold theory, we
would like to see   how these symmetries stay in the symmetry of our
3-sphere in a more direct and intuitive way, as we mentioned in the
end of the introduction.

\noindent \textbf{Remark:} In the ten examples below, the first
seven are constructed before applying orbifold theory to get the
result in Section 2 and 3, and the last three examples was
constructed after we knowing results in Sections 2 and 3. The
constructions of first nine examples use no information from
orbifold theory.

\noindent \textbf{Example 4.1} For every $g>1$, there is an abelian
group $G\cong \mathbb{Z}_2\times\mathbb{Z}_{g+1}$ which acts on
$\Sigma_g$ such that the action extends to a $G$-action on $S^3$,
for the standard embedding of $\Sigma_g\subset S^3$. When $g$ is
even, we get a cyclic group action of order $2g+2$ on $\Sigma_g$
which extends over $S^3$.

\begin{center}
\scalebox{0.6}{\includegraphics*[0pt,300pt][453pt,576pt]{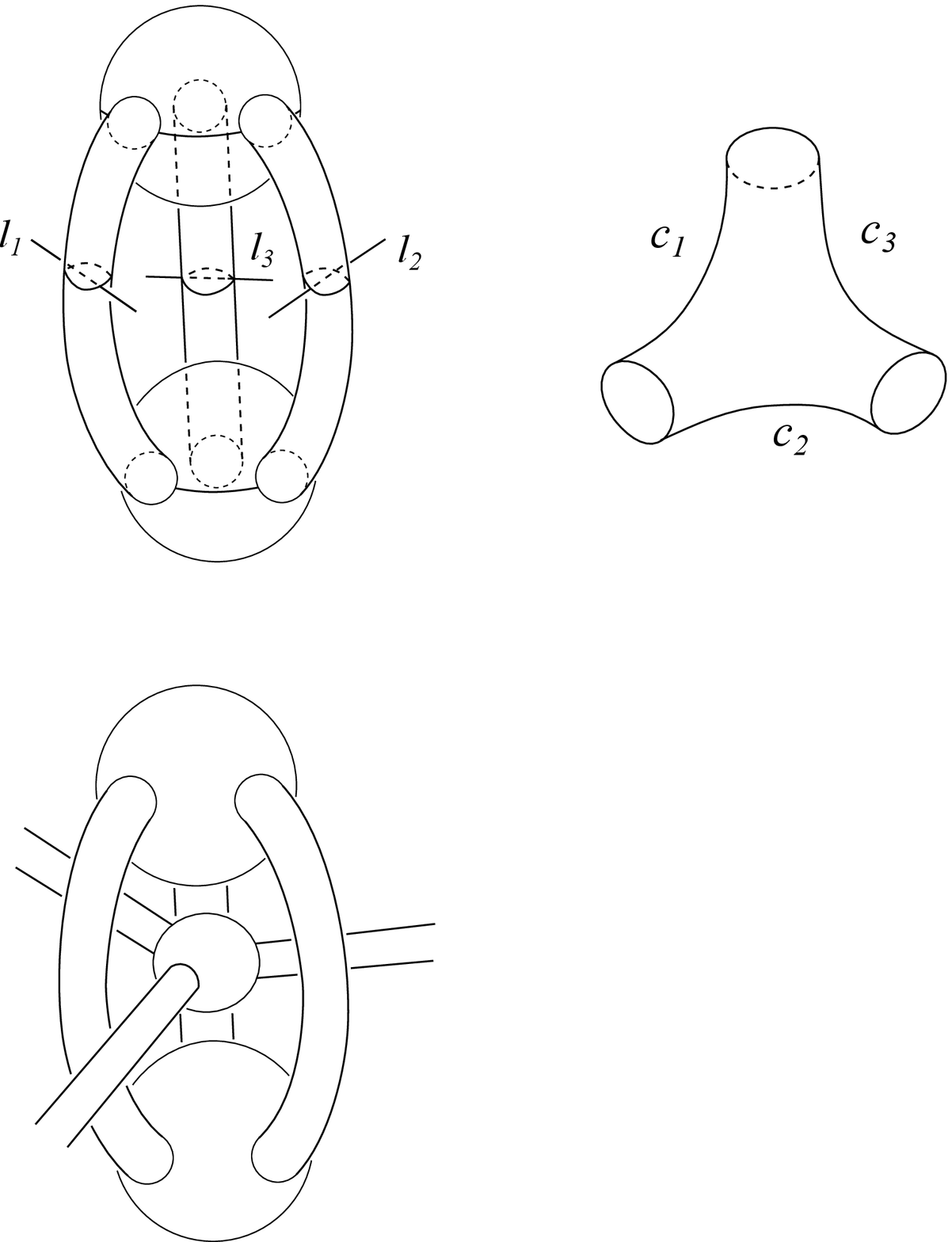}}

Figure 11
\end{center}

We embed  $V_g$ in $S^3$ as two $3$-balls together with $g+1$
handles attached to the equator, see Figure 11 for $g=2$. One can
easily see a $\mathbb{Z}_{g+1}$-action by a rotation on $S^3$ which
keeps $V_g$ invariant. We construct a $\mathbb{Z}_2$ involution on
$V_g$ as follows:  The involution restricted to each handle is a
$\pi$-rotation about each line $l_i$ drawn in Figure 11, and  the
involution maps each of the two 3-balls to the other one, without
any rotation or reflection. One can carefully check the attaching
disks to show this is well defined;  moreover the involution and the
rotation commutate and give an action of $\Z_{g+1}\times \Z_2$ on
$(S^3, V_g)$. When $g$ is even,  the composition of the involution
and the $\frac{2\pi}{g+1}$-rotation is an order $2g+2$ map.

%Hence the $2g+2$ order automorphism of $\Sigma_g$ can extend to a
%$2g+2$ order automorphism of $S^3$.

Another simple and useful way to observe this action is just to
think about a pair of pants, on the right hand side of Figure 11.
Its neighborhood in $S^3$ is a genus $2$ handlebody. This pair of
pants consists of two pieces of cloth sewn together along three
lines $c_i$. So this is just the same as on the left hand side,
letting each piece of cloth correspond to a $3$-ball and  each $c_i$
to a handle. Now the rotation on the pair of pants is quite obvious
and is extendable. The involution just changes the positions of the
two pieces of cloth, keeping $c_i$ fixed. This makes the pants inner
to outer, and because you can really do this practically, it extends
to an involution of $S^3$.

%The complement of the handlebody in $S^3$ is another handlebody. It
%is also two $3$-balls with $g+1$ handles. One ball is in the center,
%and another is at infinity, as showed in Figure 2. The automorphism
%of $S^3$ restricted on this handlebody is also an involution, a
%$\pi$-rotation on each handle and change the two balls. So this
%involution extend to an involution of $S^3$.

%\begin{center}
%\scalebox{0.5}{\includegraphics*[0pt,0pt][200pt,267pt]{fig2.eps}}
%\\
%Figure 2
%\end{center}

\noindent \textbf{Example 4.2} Now we construct an example of a
$\Z_{2g-2}$ action on $V_g$ for $g$ odd. Let us consider $S^3$ as
the unit sphere in $\mathbb{C}^2$,
$$S^3=\{(z_1,z_2)\in\mathbb{C}^2\mid |z_1|^2+|z_2|^2=1\}.$$
There is a solid torus $T\in S^3$,
$$T=\{(z_1,z_2)\in S^3\mid |z_1|\leq\frac{\sqrt{2}}{2}\}.$$
We choose $g-1$ pairs of points $a_k$ and $b_k$, $k=1, 2, \ldots,
g-1$,
$$a_k=(\frac{\sqrt{2}}{2}e^{\frac{2k\pi
i}{g-1}},\frac{\sqrt{2}}{2}e^{\frac{k\pi}{g-1}}),$$
$$b_k=(\frac{\sqrt{2}}{2}e^{\frac{2k\pi
i}{g-1}},\frac{\sqrt{2}}{2}e^{\frac{(k+g-1)\pi i}{g-1}}).$$

In the disk
$$D_k=\{(re^{\frac{2k\pi i}{g-1}},z_2)\in S^3\mid
r\geq\frac{\sqrt{2}}{2}\},$$  there is a unique diameter $\gamma_k$
connecting $a_k$ and $b_k$. Let $N_k$ be a neighborhood of
$\gamma_k$ in $S^3$; then the solid torus $T$ together with the
$(g-1)$  handles $N_k$ gives an embedding of $V_g$ into $S^3$.

\begin{center}
\scalebox{0.7}{\includegraphics*[0pt,0pt][190pt,190pt]{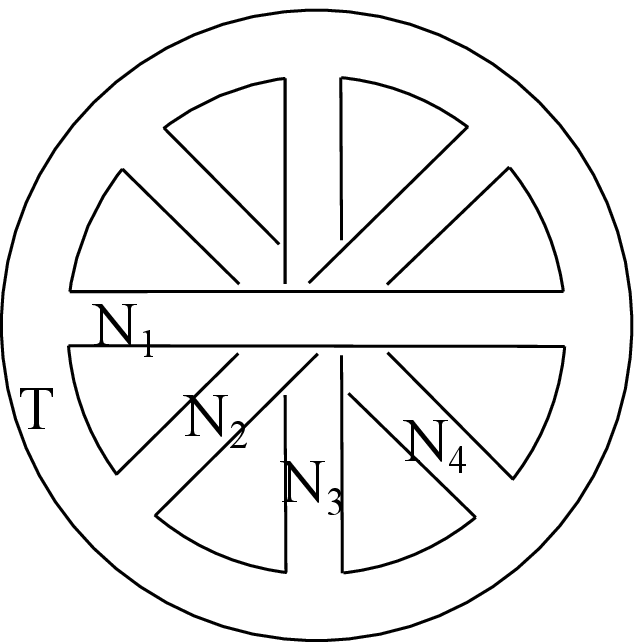}}
Figure 12
\end{center}

Now the $(2g-2)$ order action on $S^3$ can be described like this:
$$\sigma:S^3\rightarrow S^3$$
$$(z_1,z_2)\mapsto (z_1e^{\frac{2k\pi i}{g-1}},z_2e^{\frac{k\pi
i}{g-1}}).$$  This action keeps $T$ invariant and sends each $N_k$
to $N_{k+1} (mod (g-1))$. A rough picture is showed in Figure 12 for
$g=5$.

\noindent \textbf{Example 4.3} For every  $g>1$, we will construct a
group $G$ of order $4(g+1)$ which acts on $S^3=V_g\bigcup V_g'$. We
consider $V_g$ as the neighborhood of a sphere $S^2$ with $g+1$
punctured holes. We choose the holes all on the equator, centered at
the vertices of a regular $g+1$-polygon. There is a dihedral group
$D_{g+1}$ acting on $S^2$ which keeps the holes invariant (as a
set). And there is also a $\Z_2$ action changing the inner and outer
of $S^2$, as described in Example 4.1. So there is a $D_{g+1}\times
\Z_2$ action on $V_g$. This group has order $4(g+1)$.

\noindent \textbf{Example 4.4} For every square number $g>1$, we
construct a group $G$ of order $4(\sqrt{g}+1)^2$ which acts on
$S^3=V_g\bigcup V_g'$ (notice that this is greater than $4(g+1)$).

With $g=k^2$, the group $G$ is a semidirect product
$$(\Z_{k+1}\times\Z_{k+1})\rtimes_\varphi(\Z_2\times\Z_2).$$
Writing $\Z_{k+1}\times\Z_{k+1}=\langle x, y | xy=yx,
x^{k+1}=y^{k+1}=1\rangle$, $\Z_2\times\Z_2=\langle s, t | st=ts,
s^2=t^2=1\rangle$, the semidirect product is given by
$$\varphi: sxs^{-1}=y,\quad sys^{-1}=x,\quad txt^{-1}=x^{-1},\quad
tyt^{-1}=y^{-1}.$$

Consider $S^3$ as the unit sphere in $\mathbb{C}^2$
$$S^3=\{(z_1,z_2)\in\mathbb{C}^2\mid |z_1|^2+|z_2|^2=1\},$$
and let $$a_j=(e^{\frac{2j\pi i}{k+1}},0),\qquad
b_j=(0,e^{\frac{2j\pi i}{k+1}}), \qquad j=0,1,\dots,k.$$ Then the
$G$-action on $S^3$ is given by:
$$x: (z_1,z_2)\mapsto(e^{\frac{2\pi i}{k+1}}z_1,z_2),$$
$$y: (z_1,z_2)\mapsto(z_1,e^{\frac{2\pi i}{k+1}}z_2),$$
$$s: (z_1,z_2)\mapsto(z_2,z_1),$$
$$t: (z_1,z_2)\mapsto(\bar{z_1},\bar{z_2}).$$
It is easy to check this is a faithful orientation-preserving
action.

\begin{center}
\scalebox{0.3}{\includegraphics*[0pt,0pt][800pt,315pt]{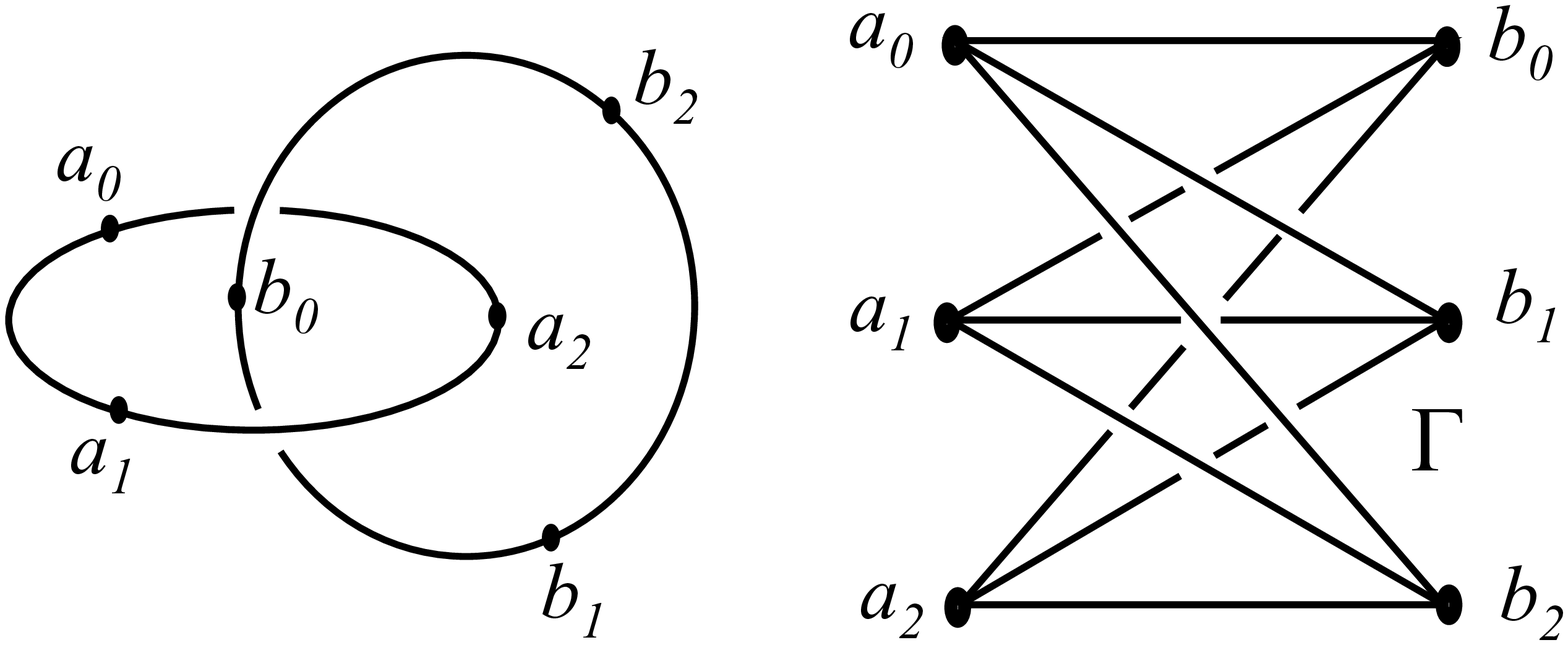}}

Figure 13
\end{center}

Notice that this $G$-action keeps the set $\{a_i, b_j\},
i,j=0,1,\dots,k $, invariant. If we join each $a_i$ and $b_j$ by a
geodesic in $S^3$, we get a two-parted graph $\Gamma\in S^3$ with
$2k+2$ vertices and $(k+1)^2$ edges. Hence $\chi(\Gamma)=-k^2+1$. A
neighborhood of $\Gamma$ in $S^3$ is a handlebody of genus $k^2=g$.
Since the $G$-action maps $\Gamma$ to itself, it induces an action
on $V_g$. Figure 13 gives the picture for $g=4$.

For $g=4$, this example gives a group action of order $36$, which is
maximal.

\noindent \textbf{Example 4.5} We view  $S^3=\{(x,y,z,w)\mid
x^2+y^2+z^2+w^2=1\}$ and  $S^2=\{(x,y,z,0)\mid
x^2+y^2+z^2=1\}\subset S^2$. For each finite group $G\leq O(3)$
which acts on $S^2\subset R^3$, we define $\widetilde{G}\leq SO(4)$
acts on $S^3$  which keeps the $S^2$ invariant as below: For each
$\sigma\in G$, define $\widetilde{\sigma}:S^3\rightarrow S^3$ by
\begin{equation*}\widetilde{\sigma}(x,y,z,w)\mapsto
\begin{cases} (\sigma(x,y,z),w) & \sigma \text{is
orientation-preserving};\\(\sigma(x,y,z),-w) & \sigma \text{is
orientation-reversing}.
\end{cases}
\end{equation*}

Now if we choose $G$ to be the symmetry group of a tetrahedron with
vertices on this $S^2$, then $G\cong S_4$. Let $\{v_1,v_2,v_3,v_4\}$
be the vertices of this tetrahedron, then $G$ and $\widetilde{G}$
keeps this vertices set invariant. If we choose $4$ holes puncturing
this $S^2$ at the positions of $\{v_1,v_2,v_3,v_4\}$, we get a
$4$-punctured sphere $X$, its neighborhood in $S^3$ is a handlebody
$V_3=N(X)$. And $\widetilde{G}$ act on $(V_3, S^3)$. Note that
$\widetilde{G}\cong G\cong S_4$, $|\widetilde{G}|=24$, this gives an
example for $g=3$ and $|\widetilde{G}|=24=12(g-1)$.

Similarly we can choose $G$ to be the symmetry group of a cube or a
dodecahedron, we get $\widetilde{G}\cong S_4 \times \Z_2$ for $g=5$
and $\widetilde{G}\cong A_5 \times \Z_2$ for $g=11$. They all
satisfies $|\widetilde{G}|=12(g-1)$.

%There are  three famous finite subgroup of $SO(4)$, the binary
%tetrahedral group $T^*$ of order 24, the binary octahedral group
%$O^*$ of order 48, the binary icosahedral group $I^*$ of order 120.
%Those groups are lifted from the symmetry groups of tetrahedron,
%%$S^3$.

\noindent \textbf{Example 4.6} The quotients $S^3/I^*$ of $S^3$ by
the binary icosahedral group is the famous Poincar\'e homology
3-sphere which is also obtained by identifying pairs of faces of the
dodecahedron. The covering $p: S^3\to S^3/I^*$ provides a
tessellation of $S^3$ by dodecahedra. Since the deck group is of
order 120 and the symmetry group of the dodecahedron is 60, it is
easy to see there is a group $G$ of order $120\times 60=7200$ acting
on this tessellation, and in particular acting on the boundary
surface $ \Sigma_g$ of a regular neighborhood of the 1-skeleton of
this tessellation. Let $v_i$ be the number of the $i$-dimensional
cells of this tessellation; then $v_0-v_1+v_2-v_3=\chi(S^3)=0$.
Clearly $v_3=120$ and $v_2=120\times 12 /2=720$. So
$v_1-v_0=v_2-v_3=600$, and therefore $g=v_1-v_0+1=601$ and $|G| =
7200 = 12(g-1)$.

%The quotients $S^3/T^*$ is obtained by identifying the pair of
%opposite faces of the octahedron. Apply the same argument to the
%coverings $S^3\to S^3/T^*$ we get a group $G$ of order $576=8(73-1)$
%acts on $(S^3, F_{73})$.

This example can be also considered as the orientation preserving
 symmetry group of the regular 120-cell in 4-space,  in the spirit
of the next example.

\noindent \textbf{Example 4.7} Let $\Delta$ be the 4-dimensional
regular Euclidean simplex and  $\Theta$ be the 4-dimensional
Euclidean cube centered at the origin of $E^4$ and  inscribed in the
unit sphere $S^3$. The radical projection of their boundaries to
$S^3$ gives two regular tessellations of $S^3$, still denoted as
$\partial \Delta$ and $\partial \Theta$ respectively.

Using the notions defined in Example 4.6, for $\partial \Delta$ we
have $v_0=v_3=5$, $v_1=v_2=10$. Since $10-5+1=6$,   the boundary
surface  of the regular neighborhood of the 1-skeleton of this
tessellation is $\Sigma_6$. Each 3-dimensional face is a tetrahedron
which has a symmetry of order 12, and then it follows that a group
$G$ of $60 = 12\times 5 = 12(g-1)$ acts on this tessellation, and in
particular acts on $\Sigma_6$.

For $\partial \Theta$, since $\Theta$ is the product of the
3-dimensional cube with an inteval, it derived easily that  we have
$v_0=16$, $v_1=32$, $v_2=24$ and $v_3=8$. Since $32-16+1=17$, the
boundary surface of the regular neighborhood of the 1-skeleton of
this tessellation is $\Sigma_{17}$. Each 3-dimensional face is a
3-dimensional cube which has a symmetry of order 24, and it follows
that a group $G$ of order $24\times 8 = 192 = 12(g-1)$ acts on this
tessellation, and in particular on $\Sigma_{17}$.

So in this example we get actions of groups on $(S^3,  \Sigma_g)$ of
maximaum order $12(g-1)$,  for $g=6$ and 17.

\noindent \textbf{Example 4.8} Let $M'=P\times S^1$, where $P$ is
the oriented pair of pants, with the induced orientation on
$\partial P=\{c_1,c_2,c_3\}$, and $S^1$ is oriented and represented
by a curve $h$. Now attach $3$ solid tori $N_i$ along the first
three boundary tori of $M'$ so that the meridian of $N_i$ is
identified with a curve of slope $l_i = 2c_i+ h$, $i=1,2,3$. Denote
the resulting manifold by $M$, which has the following properties:

(1) Recall that the full symmetry group of the pair of pants $P$ is
$G_P=D_{3}\times \Z_2$, where the $D_3$ action on $P$ is orientation
preserving, and the $\Z_2$-action, which exchanges the inner and
outer of $P$ as described in Example 4.3, is orientation revering.
We extend the $G_P=D_{3}\times \Z_2$ action on $P$ to $P\times S^1$
in an orientation preserving way by matching $D_3$ with the identity
of $S^1$ and $\Z_2$ with an orientation reversing involution (
reflection) on $S^1$. The latter extends over $M$ since it preserves
the set of attaching slopes.

(2) The map $\pi_1(P)\to \pi_1(M)$ is a surjection, therefore the
map $\pi_1(\Sigma_2)\to \pi_1(M)$ is a surjection where $\Sigma_2$
is the boundary of the regular neighborhood of $P$ in $M$; here all
maps are induced by inclusions, and we use the fact that  each
attaching curve goes over the $S^1$ direction only once.

(3) $M$ is a spherical 3-manifold with $\pi_1(M)=\Z_3\times D^*_8$,
where $D^*_8$ is the quaternion group of order 8, see \cite{Or}, in
particular $|\pi_1(M)| = 24$.

Now consider the covering $p: S^3\to M$. By (2) and Lemma
\ref{connected} the preimage of $\Sigma_2$ is connected, therefore
it is the surface $\Sigma_{25}$ by (3) which is invariant under the
group of order $24\times 12 = 12(g-1)$.

\noindent\textbf{Example 4.9}  We perform $-1$ surgery on each
component of the link given in Figure 14. considering the Wirtinger
presentation, the fundamental group of the resulting manifold is
$<x, y, z|x=yz, y=zx, z=xy>$ which is isomorphic to the quaternion
group $Q$ by the map induced by $x\mapsto i, y\mapsto j, z\mapsto
k$. The resulting manifold is the quaternion manifold $S^3/Q$.
Consider a point in front of the paper and a point behind. Through
each of the three components of the link  choose  a string
connecting these two points. We get a $\theta$-graph. The boundary
of a neighborhood of the $\theta$-graph is the surface $\Sigma_2$.
As described before, there is a  group of order 12 acting on $S^3$
which leaves this surface invariant. This action leaves also the
link invariant and extends to the surgered solid tori. Lifting to
$S^3$ we have a
 group action of order 96. Since $\pi_1(\theta)$ is generated by
$x^{-1}z$ and $x^{-1}y$, it is easy to see the homomorphism
$\pi_1(\theta)\rightarrow \pi_1(S^3/Q)$ is surjective. By  Lemma
\ref{connected}, the lift of $\Sigma_2$ is connected and hence is
$\Sigma_9$. Hence we get an extendable group action on $\Sigma_9$,
of
 order $12(9-1) = 96$.

\begin{center}
\scalebox{0.7}{\includegraphics*[0pt,0pt][237pt,213pt]{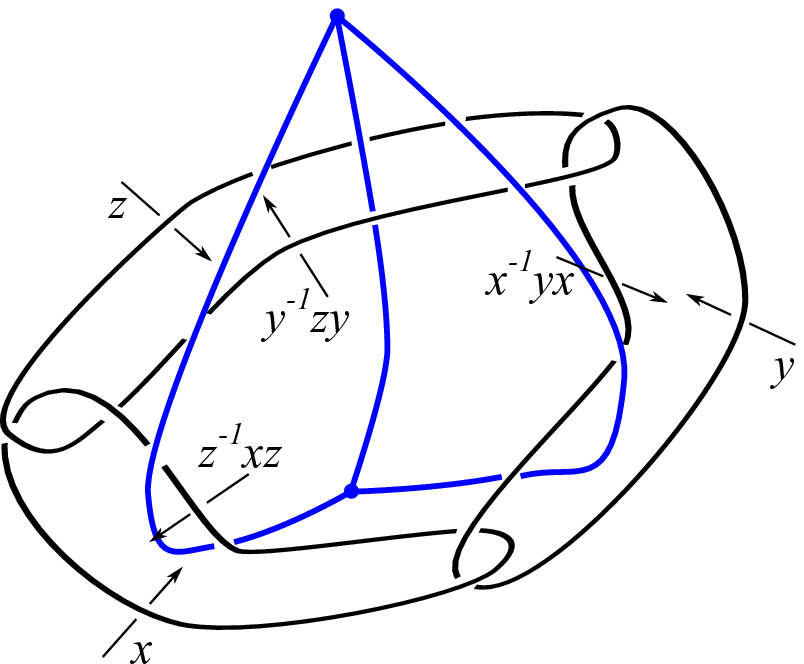}}

Figure 14
\end{center}

%\begin{center}
%\scalebox{0.7}{\includegraphics*[144pt,0pt][310pt,160pt]{fig13.eps}}\quad\quad
%\scalebox{0.7}{\includegraphics*[0pt,0pt][131pt,174pt]{fig12.eps}}

%Figure 01
%\end{center}

Similarly, performing  $+1$ surgery on a trefoil indicated in Figure
15 we obtain the Poincar\'e homology 3-sphere $P$. Letting $u=yx$,
the fundamental group has a presentation $\pi_1(P)=<u, x|
u^3=x^5=(xu)^2>$. It is isomorphic to the binary icosahedral group
$I^*$ by the map induced by $u\mapsto (1+\sigma i+\delta j)/2,
x\mapsto (\sigma+i-\delta k)/2$, where $\sigma=(\sqrt{5}+1)/2,
\delta=(\sqrt{5}-1)/2$. We can similarly construct a $\theta$-graph,
and $\pi_1(\theta)$ is generated by $x^{-1}ux^{-1}\mapsto
(\sigma+\delta j+k)/2$ and $ux^{-2}\mapsto (\sigma-\delta j+k)/2$.
One can verify that $\pi_1(\theta)\rightarrow \pi_1(P)$ is
surjective. Then we lift the neighborhood boundary surface of the
$\theta$-graph and the group action as before and obtain an
extendable action on $\Sigma_{121}$ of order $120 \times 12 = 1440 =
12(g-1)$.

\begin{center}
\scalebox{0.7}{\includegraphics*[0pt,0pt][195pt,180pt]{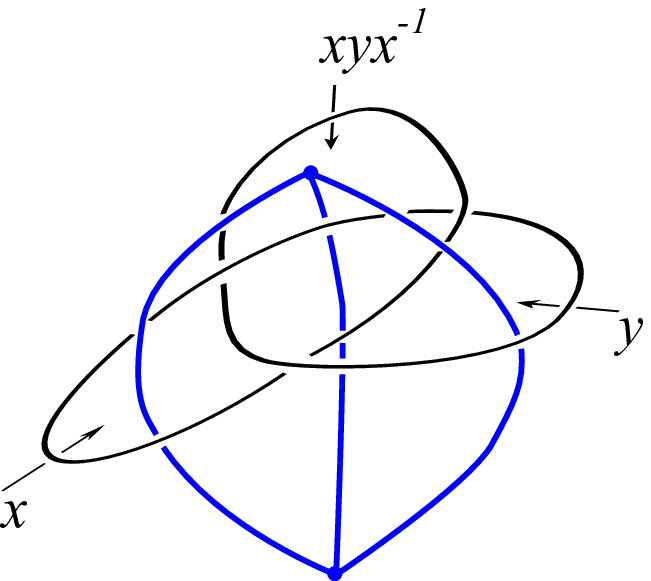}}

Figure 15
\end{center}

\noindent\textbf{Example 4.10.} Considering the 2:1 map
$SO(4)\rightarrow SO(3)\times SO(3)$, the preimage of $O\times O$,
denoted by $\mathbf{O\times O}$, acts on $S^3$ ($O$ denotes the
octahedral group, of order 24).  It has order
$24\times24\times2=1152=12(97-1)$ and the pre-fundamental domain is
a truncated cube \cite{Du2}. If two such domains are adjacent via an
octagon, we draw an edge between the centers of them. Then we get a
graph, and the boundary of the regular neighborhood of this graph is
a surface $\Sigma_{97}$ with an action of $\mathbf{O\times O}$ which
is obviously extendable.

The preimage of $O\times J$, denoted by $\mathbf{O\times J}$, acts
on $S^3$. It has order $24 \times 60 \times 2 = 2880 = 12(241-1)$
and the pre-fundamental domain is a `twice truncated tetrahedron'(or
small tetrahedron in Dunbar's paper). If two such domains are
adjacent via a dodecagon, we draw an edge between their centers.
Then we get a graph, and the boundary of a regular neighborhood of
this graph is a surface $\Sigma_{241}$ with an extendable action of
$\mathbf{O\times J}$.

\end{document}